\documentclass[11pt,reqno]{amsart}

\usepackage{times}
\usepackage{graphicx}
\usepackage{amssymb}
\usepackage{latexsym}
\usepackage{enumerate}
\calclayout
\theoremstyle{plain}
\numberwithin{equation}{section}
\newtheorem{thm}{Theorem}[section]
\newtheorem{theorem}[thm]{Theorem}

\newtheorem{example}[thm]{Example}

\newtheorem{problem}[thm]{Problem}

\begin{document}

\title{A Special Theorem Related to the Fagnano's Problem}
\author{Jun Li}
\address{
School of Science\\
Jiangxi University of Science and Technology\\ Ganzhou\\
341000\\
China.
}
\email{junli323@163.com}

\begin{abstract}
A special theorem related to the Fagnano's problem is proved and an example of the theorem is shown in a golden rectangle.
\end{abstract}

\date{2016.6}
\maketitle

\section{A special theorem related to the Fagnano's problem}
The Fagnano's problem (see, e.g., \cite{cte1}) is an optimization problem that was first stated by Giovanni Fagnano in 1775:

\begin{problem}\label{pr}
For a given acute triangle $\triangle{ABC}$, determine the inscribed triangle of minimal perimeter.
\end{problem}

The answer to Problem \ref{pr} is that the orthic triangle of $\triangle{ABC}$ has the smallest perimeter. Here, we have a special Theorem \ref{fag1} related to the problem.
\begin{figure}[ht]
\centering
\includegraphics[width=0.5\textwidth]{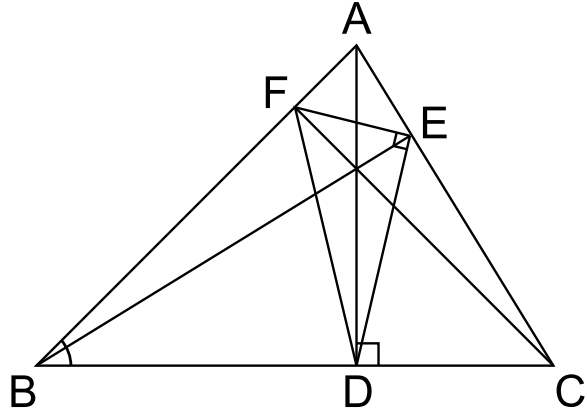}
\caption{A special theorem related to the Fagnano's problem}
\label{fig:fg6a}
\end{figure}
\begin{theorem}\label{fag1}
In Figure \ref{fig:fg6a}, Suppose $\triangle{ABC}$ is an acute-angled triangle, then, the of smallest perimeter triangle $\triangle{DEF}$ can be inscribed in $\triangle{ABC}$ is a right triangle if and only if $\triangle{ABC}$ has only one $\frac{\pi}{4}$ angle. In addition, the right angle $\angle{E}$ and the $\frac{\pi}{4}$ angle are opposite angles, which means here, $\angle{B}=\frac{\pi}{4}$.
\end{theorem}
\begin{proof}
According to the answer above, the $\triangle{DEF}$ is the orthic triangle of $\triangle{ABC}$, and it's also known that the incenter of the orthic triangle $\triangle{DEF}$ is the orthocenter of $\triangle{ABC}$ (see, e.g., \cite{cte2}). Now, suppose $\triangle{DEF}$ is a right triangle with $\angle{E}=\frac{\pi}{2}$, 
then in $\triangle{DEF}$, we have
$$\angle{D}+\angle{E}+\angle{F}=\pi$$
$$\angle{DFC}=\angle{CFE}, \angle{FDA}=\angle{ADE}, \angle{FEB}=\angle{BED}=\frac{\pi}{4}$$
then we get $$\angle{CFE}+\angle{ADE}=\frac{\pi}{4}$$
and in the quadrilateral $BDEF$, we have
$$\angle{B}+\angle{E}+\angle{BFE}+\angle{BDE}=2\pi$$
$$\angle{BFE}=\angle{BFC}+\angle{CFE}=\frac{\pi}{2}+\angle{CFE}$$
$$\angle{BDE}=\angle{BDA}+\angle{ADE}=\frac{\pi}{2}+\angle{ADE}$$
with $\angle{E}=\frac{\pi}{2}$, we conclude $\angle{B}=\frac{\pi}{4}$, and obviously, an acute-angled triangle can only have one $\frac{\pi}{4}$ angle, and also, it's easy to see that the right angle $\angle{E}$ and $\angle{B}$ are opposite angles.
\end{proof}
\begin{figure}[ht]
\centering
\includegraphics[width=0.5\textwidth]{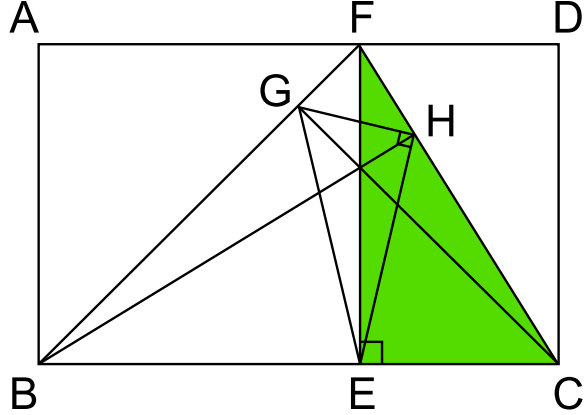}
\caption{An example of the special theorem in a golden rectangle}
\label{fig:fg6b}
\end{figure}
Next, we show an example of the special Theorem \ref{fag1} in a golden rectangle\cite{cte8}\cite[p. 274]{cte5}\cite[p. 115]{cte9}.
\begin{example}
In Figure \ref{fig:fg6b}, $ABCD$ is a golden rectangle with $AB=1$, $BC=\phi$, made up of a unit square $ABEF$ and a small golden rectangle $FECD$, then the of smallest perimeter triangle $\triangle{GHE}$ that can be inscribed in the acute-angled triangle $\triangle{BFC}$ is a right triangle having sides proportional to $(1, 2, \sqrt{5})$\cite{cte3}.
\end{example}
\begin{proof}
It's easy to see that $\triangle{BFC}$ is an acute-angled triangle with $\angle{B}=\frac{\pi}{4}$, and according to Theorem \ref{fag1}, the orthic triangle $\triangle{GHE}$ is a right triangle with $\angle{H}=\frac{\pi}{2}$. Since $\triangle{BGC} \sim \triangle{BEF}$, we have $\frac{BG}{BC}=\frac{BE}{BF}=\frac{1}{\sqrt{2}}$ and get $BG=\frac{\phi}{\sqrt{2}}$, then apply the law of cosines to $\triangle{GBE}$, we have $$GE^2=BG^2+BE^2-2{BG}\cdot{BE}\cos{\angle{B}}$$ and get $GE=\sqrt{\frac{1+{\phi}^2}{2{\phi}^2}}$. Using the similar method, we can get $HE=\sqrt{\frac{2}{1+{\phi}^2}}$, hence $\frac{GE}{HE}=\frac{\sqrt{5}}{2}$.
\end{proof}



\medskip

\noindent Mathematics Subject Classification (2010).  51M04, 11B39

\noindent Keywords.  Fagnano's problem, Golden rectangle


\begin{thebibliography}{99}

\bibitem{cte1}
Fagnano's problem, \textit{Wikipedia, The Free Encyclopedia}, \url{https://en.wikipedia.org/wiki/Fagnano_problem}.
\bibitem{cte2}
Altitude (triangle), \textit{Wikipedia, The Free Encyclopedia}, \url{https://en.wikipedia.org/wiki/Altitude_(triangle)#Orthic_triangle}.
\bibitem{cte3}
C. Radin, The pinwheel tilings of the plane, \textit{Annals of Mathematics}. \textbf{139.3} (1994) 661--702.
\bibitem{cte5}
T. Koshy, \textit{Fibonacci and Lucas Numbers with Applications}, A Wiley-Interscience Publication, 2001.
\bibitem{cte8}
M. Bicknell and V. E. Hoggatt, Golden triangles, rectangles, and cuboids, \textit{The Fibonacci Quarterly}, \textbf{7.1} (1969) 73--91.
\bibitem{cte9}
A. S. Posamentier and I. Lehmann, \textit{The Fabulous Fibonacci Numbers}, Prometheus Books, 2007.

\end{thebibliography}
\end{document}